\documentclass[reqno]{amsart}
\usepackage[english]{babel}
\usepackage[T1]{fontenc}
\usepackage{mathtools,amssymb,verbatim,dsfont}
\usepackage{physics}
\usepackage{tikz-cd}
\usetikzlibrary{babel}
\usepackage{bm}
\usepackage{float}
\usepackage{enumerate}
\usepackage{xfrac}
\usepackage{csquotes}
\usepackage[colorlinks=true]{hyperref}
\usepackage{caption}
\usepackage{etoolbox}
\usepackage{graphicx}
\usepackage{algorithm}
\usepackage{algpseudocode}
\usepackage{lineno}

\makeatletter
\patchcmd{\@setauthors}{\MakeUppercase}{}{}{}
\makeatother

\newtheorem{theorem}{Theorem}[section]
\newtheorem{lemma}[theorem]{Lemma}
\newtheorem{prop}[theorem]{Proposition}
\newtheorem{cor}[theorem]{Corollary}

\theoremstyle{definition}
\newtheorem{definition}[theorem]{Definition}

\theoremstyle{remark}
\newtheorem*{remark}{Remark}

\numberwithin{equation}{section}

\DeclareMathOperator*{\arginf}{arginf}

\DeclareMathOperator{\pr}{pr}

\DeclareMathOperator{\Cov}{Cov}

\DeclareMathOperator{\dist}{dist}
\DeclareMathOperator{\vol}{dVol}

\DeclareMathOperator{\eqdist}{\overset{d}{=}}

\DeclareMathOperator{\D}{K}

\newcommand{\ddf}[3]{\frac{\dd^{#1} #2}{\dd #3^{#1}}}

\newcommand{\innerprod}[1]{\left\langle #1 \right\rangle}

\definecolor{vividburgundy}{rgb}{0.62, 0.11, 0.21}

\begin{document}

\author{Jordi-Llu\'{i}s Figueras}

\author{Aron Persson}
%
\author{Lauri Viitasaari}
%
%
\title{On parameter estimation for $N(\mu,\sigma^2 I_3)$ based on projected data into $\mathbb{S}^2$}

\begin{abstract}
We consider the projected normal distribution, with isotropic variance, on the 2-sphere using intrinsic statistics. We show that in this case, the expectation commutes with the projection and that the covariance of the normal variable has a 1-1 correspondence with the intrinsic covariance of the projected normal distribution. This allows to estimate, after model identification, the parameters of the underlying normal distribution that generates the data. 
\end{abstract}

\maketitle

\vskip0.3cm
{\bf 2010 AMS Classification Numbers: 62H11 (Primary), 62F10 (secondary)} 

\tableofcontents

\section{Introduction}
Directional or spherical statistics form a relevant subfield of statistics in which one studies directions, rotations, and axes. A typical situation in spherical statistics is that observations are gathered on a sphere, say $\mathbb{S}^2$, and consequently the methods have to be adapted to non-Euclidean geometry. More generally, one can think of observations on more general compact Riemannian manifolds. Application areas are numerous as, for example, one can think of $\mathbb{S}^2$ representing the earth, and measurements are then observations on the surface. To simply name a few application areas, see \cite{lang2012mpg,sveier2018pose} for navigation/control in robotics, \cite{nunez2015bayesian,Nunez-Gutierrez-Escarela-2011} for modelling of wind directional changes, \cite{gao2007evidence} for finding lymphoblastic Leukemia cells,  \cite{chang1993spherical} for movement of tectonic plates, \cite{levitt1976simplified} for modelling of protein chains, and \cite{hussien2014multi,pitaval2013joint,seddik2017multi} for radiology applications in the context of MIMO-systems. For details on spherical statistics, we refer to the monograph \cite{Fisher-Lewis-1993}. 

One of the central problems in statistics is to estimate model parameters from observations. In spherical context, this can mean for example that one assumes a parametric distribution $P_\theta$ on $\mathbb{S}^2$ and then uses the data to estimate the unknown $\theta$. The most commonly applied distributions on sphere are Von Mises-Fisher distributions (also called Von Mises distribution in $\mathbb{S}$ or Kent distribution in $\mathbb{S}^2$) that can be viewed as the equivalent of normal distribution on the sphere. For the parameter estimation for Von Mises-Fisher distributions, see \cite{Tanabe-etal-2007}. Another widely applied distribution on $\mathbb{S}^d$ is the projected normal distribution. That is, the distribution of $X/\Vert X\Vert$ for $X\sim N(\mu,\Sigma)$. While the density function of the projected normal is well-known (see \cite{hernandez2017general}), the parameter estimation is much less studied due to the complicated nature of the density. In particular, the parameter estimation is studied in the circular case ($\mathbb{S}$), see \cite{Nunez-Gutierrez-2005,Nunez-Gutierrez-Escarela-2011,Presnell-etal-1998,Wang-Gelfand-2013}.

In this article, we consider the problem of parameter estimation related to the projected normal distribution onto $\mathbb{S}^2$. In contrast to the existing literature, we do not consider estimation of the parameters of the projected normal that can be obtained via standard methodology as the density is completely known. Instead, our aim is to extract information on the underlying normal distribution $N(\mu,\Sigma)$ (with support on the ambient Euclidean space $\mathbb{R}^3$) based solely on data consisting of projected points on $\mathbb{S}^2$. Obviously, we immediately run into identifiability issues if we observe only $X / \Vert X\Vert$ instead of $X \sim N(\mu,\Sigma)$. First of all, it is clear that one cannot identify arbitrary shapes of $\Sigma$ (for example, the distribution can be arbitrarily spread on the direction $\mu$ and this cannot be observed as all points in this direction are equally projected). For this reason we assume, for the sake of simplicity, isotropic variance $\Sigma =\sigma^2 I_3$ and assume $X \sim N(\mu, \sigma^2 I_3)$. However, even in this case, we can only estimate the quantities $\mu / \Vert \mu\Vert$ (the direction of the location) and $\sigma^2 / \Vert \mu \Vert^2$, the reason being that the distribution of the projection $\pr(aX)$ is invariant on $a>0$. This is indeed natural, as intuitively one can only estimate the direction $\mu / \Vert \mu\Vert$ from the projections (onto $\mathbb{S}^2$). Similarly, $N(\mu,\sigma^2 I_3)$ for larger $\sigma$ located in a distant $\mu$ seems similar to normal distribution with smaller variance but located closer, if one is only observing both distributions on the surface of a sphere. We also note that estimation of the direction $\mu / \Vert \mu\Vert$ is already well-known, and we claim no originality in this respect. Instead, our main contribution is the estimation of $\lambda= \sigma^2 / \Vert \mu \Vert^2$. For this, we study the covariance matrix of $X/\Vert X\Vert$ on $\mathbb{S}^2$ and by linking it to certain special functions and analysing their series expansions, we show that there is bijective mapping between $\lambda$ and the covariance matrix of $X/\Vert X\Vert$ on $\mathbb{S}^2$. As the latter can be estimated by using the methods of spherical statistics, we obtain the a consistent estimator for $\lambda$ via the inverse mapping that can be computed, e.g. via bisection method. 

The rest of the article is organised as follows. In Section \ref{sec:main} we present and discuss our main results. We begin with Section \ref{sec:setting} where we introduce our setup and notation. After that, we discuss convergence of sample estimators (for mean and covariance) in a context of general Riemannian manifold. The case of projected normal in $\mathbb{S}^2$ is then discussed in Section \ref{sec:S2}. All the proofs are postponed to Section \ref{sec:proofs}.

\section{Main result}
\label{sec:main}
In this section we present and discuss our main results. 
First we shall introduce some notation and clarify some terminology used in the context of manifold-valued random variables. Section \ref{sec:parest} discuss convergence of sample covariances on a general compact manifold, while the special case of $\mathbb{S}^2$ and projected normal distribution is treated in Sections \ref{sec:S2}.

\subsection{Notation and General Setting}
\label{sec:setting}
Let $M$ be a smooth and compact $n$-dimensional manifold embedded in $\mathbb{R}^k$. At each point the tangent space, $T_x M$, is imbued with the Euclidean inner product (Riemannian metric), inherited from the ambient space $\mathbb{R}^k$ making $M$ an isometrically embedded manifold as a \textit{Riemannian manifold}. The Riemannian metric, denoted $\innerprod{ \cdot,\cdot}$, induces a notion of length along smooth curves, $\gamma:[a,b]\rightarrow M$ by the formula
\[
L(\gamma) = \int_a^b \norm{\dot{\gamma}(t)} \dd t.
\]
Given two points, or equivalently a starting point and starting velocity, the curve which minimizes this distance is called a \textit{geodesic}. Centered at a point $x\in M$, the map $\exp_x:T_xM\rightarrow M$ is called the Riemannian exponential map centered at $x\in M$. Given a tangent vector $v_x\in T_xM$ let $\gamma$ be the unit speed geodesic such that $\gamma(0)=x$ and $\dot{\gamma}(0)=\frac{v_x}{\norm{v_x}}$, then
\[
\exp(v_x)=\gamma(\norm{v_x}).
\]
We shall assume that $M$ is \textit{geodesically complete} meaning that the exponential map is defined on the whole $T_x M$ for all $x\in M$. Denote the inverse of $\exp_p$ (restricted to all points $y$ for which there is a unique geodesic connecting $y$ to $p$) by $\log_p$. The distance function, $\dist: M\times M\rightarrow [0,\infty)$ (which is a topological metric) is then defined as $\dist(x,y)=\norm{\log_x(y)}$. The Riemannian metric also gives a way to measure the volume (and orientation) of a parallelotope inside the tangent space, i.e. a rescaling of the determinant from linear algebra. This generalized determinant shall be referred to as $\vol_M$, or the Riemannian volume form, which locally looks like the Lebesgue measure. By \textit{vector field}, we mean a smooth assignment of a point $x\in M$ to a tangent vector $T_x M$, and the space of vector fields is denoted by $\mathfrak{X}(M)$. For a given smooth function $f\in C^\infty(M)$, a vector field $X$ acts as a derivation on $f$ in the following sense. Take a smooth curve $(-\varepsilon,\varepsilon)\rightarrow M$ such that $\gamma(0)=p$, $\dot{\gamma}(0)=X(p)$. Then the function $X(f):M\rightarrow \mathbb{R}$ is point-wise defined by $X(f)(p)=(f\circ \gamma)'(0)$. 

Note that for given vector fields $X,Y:M\rightarrow TM$, the expression $\innerprod{X,Y}$ is smooth function $M\rightarrow \mathbb{R}$. In order to differentiate vector fields a notion of \textit{connection} is required. Here we shall use the \textit{Levi-Civita connection} $\nabla:\mathfrak{X}(M)\times \mathfrak{X}(M)\rightarrow \mathfrak{X}(M)$ which is uniquely defined by the following identities:
\begin{enumerate}[i)]
    \item $\nabla_{fX+gY} Z= f\nabla_X Z+g\nabla_Y Z$ for all smooth functions $f,g\in \mathcal{C}^\infty(M)$ and all smooth  vector fields $X,Y,Z\in \mathfrak{X}(M)$;
    \item $\nabla_{X} (gY+hZ)= g\nabla_X Y +  X(g) Y +h \nabla_X Z +  X(h) Z $ for all smooth functions $g,h\in \mathcal{C}^\infty(M)$ and all smooth vector fields $X,Y,Z\in \mathfrak{X}(M)$;
    \item $X(\innerprod{Y,Z})=\innerprod{\nabla_X(Y),Z}+\innerprod{Y,\nabla_X(Z)}$ for all smooth vector fields $X,Y,Z\in \mathfrak{X}(M)$;
    \item $(\nabla_X Y-\nabla_Y X)(f)= X(Y(f))-Y(X(f))$ as derivations for all smooth functions $f\in \mathcal{C}^\infty(M)$ and all smooth vector fields $X,Y\in \mathfrak{X}(M)$.
\end{enumerate}
Lastly, let $P_{q,\ell}:T_{q}M \rightarrow T_\ell M$ denote the parallel transport from $q$ to $\ell$. That is, $P_{q,\ell}(v_{q})$ is the unique point $\tau(\dist(q,\ell)) \in T_\ell M$ that satisfies the initial value problem
\[
\nabla_{\dot{\gamma}} \tau(t)=0, \qquad \tau(0) =v_{q},
\]
where $\gamma:[0,\dist(q,\ell)]\rightarrow M$ is the geodesic connecting $q$ to $\ell$ and $\tau$ is a vector field along $\gamma$, i.e. $\tau(t)= \tau_{\gamma(t)}$. It is worthwhile to consider the parallel transport of the Riemannian logarithm map
\[
q \longmapsto P_{q,\ell}\log_{q}(\nu), 
\]
where $\nu,\ell \in M$ are fixed points. This map has Taylor expansion around $q= \ell$ given by
\begin{equation}
\label{eq:Taylor}
P_{q,\ell}\log_{q}(\nu) = \log_\ell(\nu) + \nabla_{\log_\ell(q)} \log_\ell(\nu) + \mathcal{O}(\dist(q,\ell)^2).
\end{equation}

Let $S\in \mathbb{R}^k$ be a dense subset and let $\pr:S\rightarrow M$ be the projection onto $M\subseteq \mathbb{R}^k$, assumed smooth everywhere on $S$ and hence almost everywhere in $\mathbb{R}^k$. Let $(\Omega,\mathbb{P})$ be a probability space and consider a normally distributed (multivariate) random variable
\[
X: \Omega \rightarrow \mathbb{R}^k
\]
such that
\[
\mathbb{P}(X\notin S)= 0.
\]
Let $x_1,\dots ,x_L$ be an independently drawn sample of $X$. Suppose now that we observe
\[
\pr(x_1) ,\dots ,\pr(x_L).
\]
The question is now can one estimate the mean and covariance of this projected sample?

In order for this question to have meaning a notion of mean and covariance is needed for manifold-valued random variables. The intrinsic mean, aka Frech\'et mean, (see \cite{pennec2006intrinsic}) of an absolutely continuous random variable $X:\Omega \rightarrow M$ is defined as follows.

\begin{definition}
\label{def:frechetmean}
Let $M$ be a Riemannian manifold with corresponding distance function $\dist$ and volume form $\vol_M$ respectively. Moreover suppose $p_X:M\rightarrow [0,\infty)$ is a probability density function of some absolutely continuous random variable $X:\Omega \rightarrow M$. The expected value of $X$ is then defined by
\begin{equation}
\label{eq:arginfintegral}
\arginf_{q\in M} \int_M \dist(q,y)^2 p_{X}(y) \vol_M(y) = \mathbb{E}[X]. 
\end{equation}
\end{definition}

\begin{remark}
The argument in the infinum in \eqref{eq:arginfintegral} need not exist nor does it need to be unique. For example, consider $X\eqdist N(0,\sigma^2 I)\in \mathbb{R}^{n+1}$, then $\pr(X)$ follows the uniform distribution on $\mathbb{S}^n$ and all points in $\mathbb{S}^n$ satisfy the infinum in \eqref{eq:arginfintegral}. This is a very important difference from the case of $\mathbb{R}^n$-valued random variables. In $\mathbb{R}^n$ we know that if $X$ is absolutely continuous, integrable and square integrable, then any point $\mu \in \mathbb{R}^n$ which minimizes the least square integral
\[
\int_{\mathbb{R}^n} \abs{x-\mu}^2 p_X \dd x
\]
is unique.
\end{remark}

The intrinsic definition of the covariance matrix for a random variable on a Riemannian manifold is well known as well (see, e.g.  \cite{pennec2006intrinsic}).

\begin{definition}
\label{def:intrinsiccovar}
Let $M$ be a geodesically complete Riemannian manifold, and let $\log_q:M \rightarrow T_qM$ be the natural log map (defined a.e.). Let $X$ be an $M$-valued absolutely continuous random variable with intrinsic mean $\mu = \mathbb{E}[X] \in M$. Then, the \textit{Covariance matrix} for $X$ is the linear map $\Cov(X) : T_\mu M\rightarrow T_\mu M$ defined by the integral
\[
\Cov(X)=\int_M \log_\mu(y)\log_\mu(y)^T p_X(y)\vol_M(y).
\]
\end{definition}

\subsection{Convergence of Sample Estimators on a Compact Manifold}
\label{sec:parest}

Let $y_\ell$ be a sample of $L$ independent measurements on a compact manifold $M$. In order to estimate the mean of such an sample, we shall utilize the discrete version of Definition \ref{def:frechetmean}. This definition follows that of \cite{pennec2006intrinsic}.

\begin{definition}
\label{def:empiricalmean}
Let $\{y_\ell\}^L_{\ell=1}$ be a sample of $L$ points on a Riemannian manifold $M$. Then the \textit{empirical mean} is defined as
\[
\arginf_{q\in M} \sum_{\ell=1}^L \dist^2(y_\ell,q).
\]    
\end{definition}

Note that since the function $\dist^2$ has good regularity, one may hope to find such an infinum by solving the (non-linear) equation
\[
q: \sum_{\ell=1}^L \log_q(y_{\ell})=0.
\]
It was shown in \cite{hotz2022central} that if $M$ is compact and if $y_{\ell}$ is sampled from a distribution with a unique empirical mean $x_0$, then the distribution of the above follows a central limit theorem. More precisely, if $\mu=\mathbb{E}[y_\ell]$ is the true mean of the underlying distribution and if $\hat{\mu}$ is the empirical mean, it holds that
\[
\sqrt{L}\log_\mu\left(\hat{\mu}\right) \overset{\text{d}}{\longrightarrow} N(0,V)
\]
for some linear map $V\in \mathcal{L}(T_\mu M,T_\mu M)$. 

\begin{definition}
Let $\{y_\ell\}^L_{\ell=1}$ be a sample of $L$ points on a Riemannian manifold $M$ with an unique empirical mean $\hat{\xi}$. Then the the \textit{empirical covariance} of the sample is defined by
    \begin{equation}
        \label{eq:empiricalcovariance}
        \hat{V}=\frac{1}{L-1} \sum_{\ell=1}^L \log_{\hat{\xi}}(y_\ell)\log_{\hat{\xi}}(y_\ell)^T    
    \end{equation}
    where $\hat{V}: T_{\hat{\xi}}M\rightarrow T_{\hat{\xi}} M$ is a linear map.
\end{definition}

Since the empirical mean converges by the results of \cite{hotz2022central}, it remains to verify that the empirical covariance in Equation \eqref{eq:empiricalcovariance} converges to the covariance $V$ of the limiting distribution $ \lim_{L\to \infty}\sqrt{L}\log_\xi\left(\hat{\xi}\right)$. 

The following result shows that, in the case of isotropic covariance, the empirical covariance converges. The proof is postponed into Section \ref{sec:empiricalvarience}.
\begin{theorem}
\label{thm:empiricalvarience}
Let $\{\xi_\ell\}_{\ell=1}^L$ be $L$ independent identically distributed random variables on a compact geodesically complete manifold $M$. Suppose further they have unique mean $\mathbb{E}[\xi_\ell]=\mu$ and isotropic covariance $\Cov(\xi_\ell)=v I$. Then,
\[
\frac{1}{L-1} \sum_{\ell=1}^L \log_{\hat{\mu}}(\xi_\ell)\log_{\hat{\mu}}(\xi_\ell)^T \overset{\mathbb{P}}{\longrightarrow} v I
\]
with the same rate of convergence as the empirical mean $\hat{\mu}$ of the sample $\{\xi_\ell\}_{\ell=1}^L$. 
\end{theorem}

\begin{remark}
    As far as we know the rate of convergence for the empirical mean on compact manifolds is not completely solved as of now. In \cite{hotz2022central} has shown that the empirical mean has a rate of convergence $\sqrt{L}$ for a large class of manifolds, but not all compact, geodesically complete manifolds. However, \cite{ahidar2020convergence} provides rates of convergence for the empirical mean to a general class of metric spaces, for which it seems that the rate of convergence $\sqrt{L}$ isn't necessarily true.
\end{remark}

Following a similar method as in \cite{pennec2012exponential} the update scheme for a sample of observations are given in Algorithm \ref{alg:varestS2}.

\begin{algorithm}
\caption{Estimating intrinsic average and covariance for a sample $y_\ell$ on a manifold $M$.}
\label{alg:varestS2}
\begin{algorithmic}[1]
\State Make initial guess $\hat{\mu}_0$ for average point, take e.g. $\hat{\mu}_0=\xi_1$.
\State Compute
\[
X=\frac{1}{L}\sum_{\ell=1}^L \log_{\hat{\mu}_0}(y_\ell).
\]
\State Update estimate by setting
\[
\hat{\mu} = \exp_{\hat{\mu}_0}(X).
\]
\State Repeat step 2-3 with $\hat{\mu}_0=\hat{\mu}$ until $\norm{X}$ is small.
\State Compute sample covariance by
\[
\hat{V}= \frac{1}{L-1}\sum_{\ell=1}^{L} \log_{\hat{\mu}}(y_\ell) \log_{\hat{\mu}}(y_\ell)^T.
\]
\end{algorithmic}
\end{algorithm}

\subsection{Observing the Projected Normal in $\mathbb{S}^2$}
\label{sec:S2}

In this section we consider $M=\mathbb{S}^2$, the unit 2-sphere in $\mathbb{R}^3$. In this case, the projection map is simply
\[
\pr(x)= \frac{x}{\norm{x}}
\]
and the domain of definition for $\pr$ is $S=\mathbb{R}^3\backslash\{0\}$.

Throughout, we shall use spherical coordinates, i.e. for  classical Cartesian coordinates $(x,y,z)^T$ in $\mathbb{R}^3$ we set $x=\cos(\theta)\sin(\phi)$, $y=\sin(\theta)\sin(\phi)$, and $z=\cos(\phi)$. Consider two points $(\theta_1,\phi_1)^T$ and $(\theta_2,\phi_2)^T$. Then it is a classical result (see e.g. \cite{kells1940plane}) that the distance function may be written as
\[
\dist((\theta_1,\phi_1)^T,(\theta_2,\phi_2)^T)= \acos\left( \cos(\phi_1)\cos(\phi_2)+\sin(\phi_1)\sin(\phi_2)\cos(\theta_2-\theta_1) \right).
\]
It has been shown in \cite{hernandez2017general} that the projected normal, i.e. the random variable defined by $\pr(X)$ where $X \eqdist N(\mu,\Sigma)$, has the following probability density function 
\begin{equation}
\label{eq:sphereprobdens}
p_{\pr(X)}(\theta,\phi) = \left(\frac{1}{2\pi A}\right)^{3/2} \abs{\Sigma}^{-\frac{1}{2}} \exp(C) \left( \D+ \D^2 \frac{\Phi(\D)}{\varphi(\D)} + \frac{\Phi(\D)}{\varphi(\D)} \right)
\end{equation}
at a point $u=(\cos(\theta)\sin(\phi),\sin(\theta)\sin(\phi),\cos(\phi))^T$, where $A=u^T \Sigma^{-1}u$, $B=u^T \Sigma^{-1} \mu$, $C= -\frac{1}{2} \mu^T \Sigma^{-1} \mu$, and $\D=BA^{-\frac{1}{2}}$. Moreover, the functions $\varphi,\Phi:\mathbb{R}\rightarrow \mathbb{R}$ are defined as
\[
\varphi(x) = \exp(-\frac{x^2}{2})
\]
and
\[
\Phi(a) = \int_{-\infty}^a \varphi(x) \dd x,
\]
respectively.

In our context, this random variable is then \textit{observed} in $\mathbb{S}^2$. As $\mathbb{S}^2$ is an isometrically embedded Riemannian manifold, Definition \ref{def:frechetmean} applies for the projected normal. Intuitively the expectation of a projected normal distribution ought to be $\pr(\mu)$. Unfortunately for the fully general case of the projected normal, this conjecture will be false. However, by imposing an isotropic covariance $\Sigma$ onto $X$ we shall show that this intuition is indeed correct. This is the topic of the next result whose proof is postponed into Section \ref{seq:prooflemaverage}.

\begin{theorem}
\label{thm:average}
Let $X$ be a normally distributed random variable with average $\mu\in \mathbb{R}^3$ and covariance matrix $\Sigma$, i.e. $X\eqdist N(\mu,\Sigma)\in \mathbb{R}^3$.  If $\Sigma=\sigma^2 I_3$, then $\mathbb{E}[\pr(X)]=\pr(\mu)$.
\end{theorem}

\begin{remark}
    The above theorem is true whenever $\mu$ is an eigenvector of $\Sigma$. However, if $\mu$ is not an eigenvector of $\Sigma$, then the statement of Theorem \ref{thm:average} is false in general. To see this, let $\mu=(0,0,1)^T$ and
    \[
    \Sigma = \begin{pmatrix}
        1 & 0 & 0\\ 0 & 1 & 0.5\\ 0 & 0.5 & 1
    \end{pmatrix}.
    \]
    In this case it follows that the probability density function is not symmetric around $\mu$, see Figure \ref{fig:counterexample} below.
\end{remark}

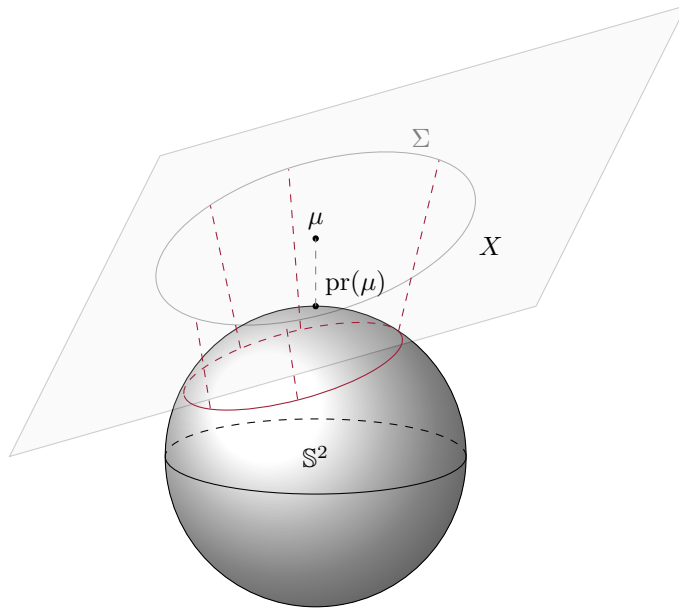
\begin{figure}[!ht]
    \centering
\begin{tikzpicture}[
  point/.style = {draw, circle, fill=black, inner sep=0.7pt},
]
\def\rad{2cm}
\coordinate (O) at (0,0); 
\coordinate (N) at (0,\rad);
\coordinate (C) at (0,1.45*\rad);

\filldraw[ball color=white] (O) circle [radius=\rad];
\draw[dashed] 
  (\rad,0) arc [start angle=0,end angle=180,x radius=\rad,y radius=5mm];
\draw
  (\rad,0) arc [start angle=0,end angle=-180,x radius=\rad,y radius=5mm];
\begin{scope}[xslant=0.5,yshift=\rad,xshift=-2]
\filldraw[fill=gray!20,opacity=0.2]
  (-4,2) -- (2,4) -- (2,0) -- (-4,-2) -- cycle;
\node at (1,0.8) {$X$};  
\end{scope}
\node[above right] at (N) {$\pr(\mu)$};
\node at (O) {$\mathbb{S}^2$};
\node[above] at (C) {$\mu$};
\node[above,color=gray] at (0.7*\rad,2*\rad) {$\Sigma$};
\node[point] at (N) {};
\node[point] at (C) {};
\draw[dashed,color=gray] (N) --(C); 
\draw[dashed,color=vividburgundy] ({-0.7*\rad},{0.32*\rad}) {} -- ({-0.8*\rad},{0.94*\rad}) {};
\draw[dashed,color=vividburgundy] ({0.55*\rad},{0.83*\rad}) {} -- ({0.82*\rad},{1.97*\rad}) {};
\draw[dashed,color=vividburgundy] ({-0.12*\rad},{0.38*\rad}) {} -- ({-0.20*\rad},{0.9*\rad}) {};
\draw[dashed,color=vividburgundy] ({-0.1*\rad},{0.85*\rad}) {} -- ({-0.18*\rad},{1.91*\rad}) {};
\draw[dashed,color=vividburgundy] ({-0.5*\rad},{0.72*\rad}) {} -- ({-0.7*\rad},{1.67*\rad}) {};
\draw[dashed,color=vividburgundy] {[rotate around={15:(-0.15*\rad,0.6*\rad)}] (0.6*\rad,0.6*\rad) arc [start angle=0,end angle=180,x radius=1.5cm,y radius=0.45cm]};
\draw[color=vividburgundy] {[rotate around={15:(-0.15*\rad,0.6*\rad)}] (0.6*\rad,0.6*\rad) arc [start angle=0,end angle=180,x radius=1.5cm,y radius=-0.45cm]};
\draw[rotate around={17:(C)},opacity=0.3] (C) ellipse (2.2cm and 1cm);
\end{tikzpicture}
    \caption{Normally distributed random variable $X\eqdist N(\mu,\Sigma)$, with covariance matrix $\Sigma$ such that the mean $\mu$ is not an eigenvector for $\Sigma$, for which $\mathbb{E}[\pr(X)]\neq \pr(\mu)$.}
    \label{fig:counterexample}
\end{figure}

In general, the tangent space of $\mathbb{S}^2$ at a point $q$ is the set of vectors orthogonal (in the Euclidean product inherited from $\mathbb{R}^3$) to $q$. By noting that the direction of the geodesic connecting a point $\mu$ and $q$ have starting velocity in the direction of $q \times \mu$ (i.e. the cross-product of $\mathbb{R}^3$). The logarithm map of $q$ centered at a point $\mu$ is thus the vector of the starting velocity of the geodesic connecting $\mu$ with $q$ times the length of the geodesic connecting $\mu$ and $q$, see \cite{li2018wavefront}. Without loss of generality, consider $\mu=(0,0,1)^T$. Then
\[
\log_\mu(q)=\begin{pmatrix}
\cos(\theta)\sin(\phi)\\ \sin(\theta) \sin(\phi)
\end{pmatrix} \frac{\phi}{\abs{\sin(\phi)}},
\]
where the basis for the vector is in a rewritten form of the basis (of $T_{\mu}\mathbb{S}^2$)
\[
\left\{\begin{pmatrix}
1\\ 0\\ 0
\end{pmatrix},  \begin{pmatrix}
0\\ 1 \\ 0
\end{pmatrix}\right\}.
\]
Hence the \textit{intrinsic} covariance of $\pr(X)$ can be written as
\[
\begin{aligned}
\Cov(\pr(X)) =& \int_0^{2\pi} \int_0^\pi \begin{pmatrix}
\cos^2(\theta)\sin^2(\phi) & \cos(\theta)\sin(\theta)\sin^2(\phi)\\
\cos(\theta)\sin(\theta)\sin^2(\phi) & \sin^2(\theta)\sin^2(\phi)
\end{pmatrix}\\
&\cdot \frac{\phi^2}{\sin^2(\phi)} \left(\frac{1}{2\pi}\right)^{3/2}  \exp(-\frac{1}{2\sigma^2})\\
&\cdot\left(\D + \D^2\frac{\Phi(\D)}{\varphi(\D)} + \frac{\Phi(\D)}{\varphi(\D)} \right) \sin(\phi)\dd\phi \dd\theta,
\end{aligned}
\]
where $\D= \frac{1}{\sigma}\cos(\phi)$. Or equivalently, by simplifying and integrating w.r.t $\theta$,
\begin{equation}
\begin{aligned}
\Cov(\pr(X)) =& \frac{\pi}{(2\pi)^{3/2}} \exp(-\frac{1}{2\sigma^2}) \begin{pmatrix}
1 & 0\\
0 & 1
\end{pmatrix}\\
&\cdot \int_0^\pi  \phi^2\left(\frac{\cos(\phi)}{\sigma} + \left( \frac{\cos^2(\phi)}{\sigma^2} + 1\right)  \frac{\Phi(\frac{\cos(\phi)}{\sigma})}{\varphi(\frac{\cos(\phi)}{\sigma})} \right) \sin(\phi)\dd\phi\\
=&: \begin{pmatrix}
    1 & 0\\0&1
\end{pmatrix}f(\sigma).
\end{aligned}
\label{eq:varianceS2}
\end{equation}

\begin{remark}
In the limit $\sigma \to \infty$ it holds that $p_{\pr(X)} \to \frac{1}{4\pi}\mathds{1}_{\mathbb{S}^2}$. The intrinsic covariance in this limit is
\[
\frac{\pi^2-4}{4}\begin{pmatrix}
1& 0\\
0 & 1
\end{pmatrix}.
\]
In the limit 
 $\sigma\to 0$ on the other hand, $p_{\pr(X)}$ converges to the point mass distribution while the covariance goes to zero.
\end{remark}


Note that the above estimates are only for the \textit{intrinsic} expectation and covariance. If we want to relate these estimates to the \textit{extrinsic} expectation and covariance we end up with the obvious problem that $\pr(X)$ and $\pr(aX)$, $a>0$, have the same distribution. Hence we need to choose which normal distribution in $\mathbb{R}^3$ corresponds to the intrinsic covariance and average. By Lemma \ref{thm:average} the average is known (up to a factor). Moreover, it turns out there is a one-to-one correspondence between the extrinsic covariance and the intrinsic covariance up to a factor $a>0$, and thus one may estimate the underlying covariance parameter $\sigma$ using the relation \eqref{eq:varianceS2} with e.g. the bisection method. 

The degeneracy of the factor $a$ essentially means that we can only estimate the parameter $\sigma$ in the case when the projected average is the \textit{true} average of the $\mathbb{R}^3$-valued normal random variable. More generally, if $X\eqdist N(\mu,\sigma^2 I_3)$, then $\frac{X}{\Vert \mu\Vert} \eqdist N\left(\frac{\mu}{\Vert \mu\Vert},\frac{\sigma^2}{\Vert \mu\Vert^2} I_3\right)$. Consequently, we can only estimate the quantities $\frac{\mu}{\Vert \mu\Vert}\in \mathbb{S}^2$ and $\frac{\sigma^2}{\Vert \mu\Vert^2}$. This means that, as expected, we can estimate the direction $\frac{\mu}{\Vert \mu\Vert}$ of the extrinsic distribution  $X\eqdist N(\mu,\sigma^2 I_3)$ but not the distance $\Vert \mu\Vert$. For the variance on the other hand, we can only estimate the quantity $\frac{\sigma^2}{\Vert \mu\Vert^2}$. This is expected as well, as for the distribution $ N(\mu,\sigma^2 I_3)$ located far away ($\Vert \mu\Vert$ large) the projections onto $\mathbb{S}^2$ are not wide spread even if $\sigma>0$ would be large.

The required one-to-one correspondence between extrinsic and intrinsic covariances is formulated in the following proposition, whose proof is presented in Section \ref{sec:bijection}. The statement is also illustrated  in Figure \ref{fig:varbij}.
\begin{prop}
\label{prop:bijection}
Let $X$ be a normally distributed random variable with mean $\mu\in \mathbb{R}^3$ and isotropic variance $\sigma^2 I_3$, i.e. $\frac{X}{\norm{\mu}}\eqdist N(\frac{\mu}{\norm{\mu}},\frac{\sigma^2}{\norm{\mu}^2} I_3)$, then $\Cov (\pr(X))= vI_2$ for $0\leq v< \frac{\pi^2 -4}{4}$ and the relation
\[
f\left( \frac{\sigma^2}{\norm{\mu}^2}\right) := \frac{\trace(\Cov(\pr(X)))}{2} =v
\]
is a bijection.
\end{prop}
As a consequence of Proposition \ref{prop:bijection} and Theorem \ref{thm:empiricalvarience} we can conclude that given measurements on the sphere, we can estimate the scalar variance of an isotropically distributed normal random vector in $\mathbb{R}^3$. This leads to the next result that can be viewed as the main theorem of the present paper. Its proof follows directly from Theorem \ref{thm:empiricalvarience} and Proposition \ref{prop:bijection}. The rate of convergence is immediate from noting that \cite[Theorem 2]{hotz2022central} applies to $\mathbb{S}^2$, and thus the empirical mean has rate of convergence $\sqrt{L}$, and by Theorem \ref{thm:empiricalvarience} so does the empirical covariance.

\begin{theorem}
\label{thm:estimatecovar}
    Let $X$ be a normally distributed random variable with mean $\mu\in \mathbb{R}^3$ and isotropic variance $\sigma^2 I_3$, i.e. $X\eqdist N(\mu,\sigma^2 I_3)$. Given independent measurements $(x_1,x_2,\dots, x_L)$ from $\pr(X)$, we can estimate $\lambda=\frac{\sigma^2}{\norm{\mu}^2}$ by 
    \[
    \hat{\lambda}=f^{-1}\left(\frac{\trace(\hat{V})}{2}\right)
    \]
    where $\hat{V}$ is the empirical covariance matrix given in Equation \eqref{eq:empiricalcovariance} and where $f$ is the bijection given in Proposition \ref{prop:bijection} and defined in Equation \eqref{eq:varianceS2}. Moreover it holds that
    \[
    \hat{\lambda} \overset{\mathbb{P}}{\longrightarrow} \frac{\sigma^2}{\norm{\mu}^2}
    \]
    as $L\to \infty$, with rate of convergence $\sqrt{L}$.
\end{theorem}

\begin{figure}[!ht]
\centering
\includegraphics[width=0.7\textwidth]{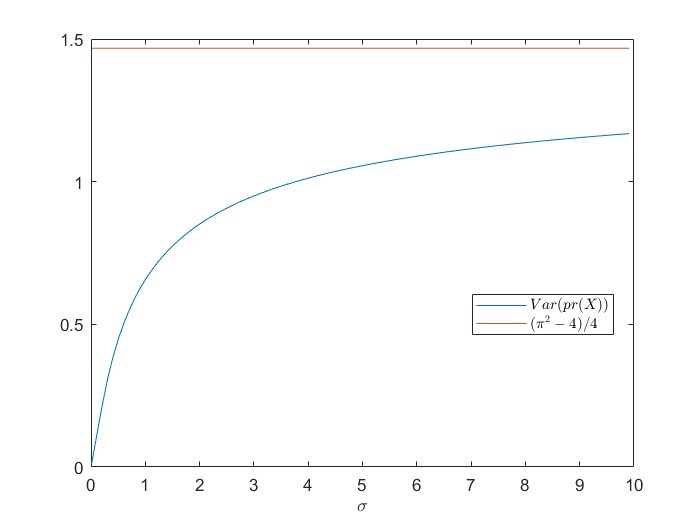}
\caption{A plot of the (scalar) variance of $\pr(X)$ where $X\eqdist N(\mu,\sigma^2 I_3)$ with $\mu\in \mathbb{S}^2$. The red line indicates the upper bound $(\pi^2-4)/4 =\lim_{\sigma\to \infty} \Cov(\pr(X))$.}
\label{fig:varbij}
\end{figure}
We conclude this section with  Table \ref{tab:simofvar}, illustrating an empirical verification of Theorem \ref{thm:estimatecovar}. The simulations are conducted by generating $L$ samples from $X\eqdist N(\mu,\sigma^2 I_3)$ with $\mu(0,0,1)^T\in \mathbb{S}^2$ and $\sigma=1$. The sample covariance is computed as in Algorithm \ref{alg:varestS2}, and its error is then compared to the theoretical covariance. 
\begin{table}[h]
\centering
\begin{tabular}{|l|l|l|l|l|l|l|l|l|}
\hline
L     & 30                 & 50                  & 100                 & 1000                & $10^4$ & $10^5$ & $10^6$              \\ \hline
error & 0.013 & 0.0080 & 0.0055 & 0.0033 & 0.0015 & 8.4e-05 & 2.9e-05 \\ \hline
\end{tabular}
\caption{The absolute error of $100$ times repeated Monte-Carlo empirical covariance compared to the theoretical covariance of the projected normal distribution using $L$ data points and $\sigma=1$.}
\label{tab:simofvar}
\end{table}

\section{Proofs}
\label{sec:proofs}
\subsection{Proof of Theorem \ref{thm:empiricalvarience}}
\label{sec:empiricalvarience}
First, note that
\[
\frac{1}{L-1} \sum_{\ell=1}^L \log_{\hat{\mu}}(\xi_\ell)\log_{\hat{\mu}}(\xi_\ell)^T
\]
is a linear map $T_{\hat{\mu}} \mathbb{S}^2\rightarrow T_{\hat{\mu}} \mathbb{S}^2$ but $\Cov(\xi)$ is by definition a linear map $T_{\mu} \mathbb{S}^2\rightarrow T_{\mu} \mathbb{S}^2$. In order to compare $\hat{V}$ and $\Cov(\xi)$ we shall parallel transport the linear map $\hat{V}: T_{\hat{\mu}}\rightarrow T_{\hat{\mu}}$ to a linear map $T_{\mu} \mathbb{S}^2\rightarrow T_{\mu} \mathbb{S}^2$. That is we look at $P_{\hat{\mu},\mu}\hat{V} P_{\mu,\hat{\mu}}$ and see how far away it is from $vI_2$.
It follows from Equation \eqref{eq:Taylor} that
\[
P_{\hat{\mu},\mu} \log_{\hat{\mu}}(\xi_\ell)\log_{\hat{\mu}}(\xi_\ell)^T P_{\mu,\hat{\mu}}
\]
has first order expansion
\[
\begin{aligned}
P_{\hat{\mu},\mu} \log_{\hat{\mu}}(\xi_\ell)\log_{\hat{\mu}}(\xi_\ell)^T P_{\mu,\hat{\mu}} &=\log_\mu(\xi_\ell)\log_\mu(\xi_\ell)^T+\left( \nabla_{\log_\mu(\hat{\mu})} \log_\mu(\xi_\ell) \right)\log_\mu(\xi_\ell)^T\\
&+ \log_\mu(\xi_\ell)\left( \nabla_{\log_\mu(\hat{\mu})} \log_\mu(\xi_\ell) \right)^T+ \mathcal{O}(\dist(\mu,\hat{\mu})^2).
\end{aligned}
\]
Therefore, 
\[
\begin{aligned}
\lim_{n\to \infty} &\norm{\mathbb{E}\left[P_{\hat{\mu},\mu}\frac{1}{n-1} \sum_{\ell=1}^n \log_{\hat{\mu}}(\xi_\ell)\log_{\hat{\mu}}(\xi_\ell)^T P_{\mu,\hat{\mu}} -vI\right]}\\
\leq&\lim_{n\to \infty}\norm{ \frac{1}{n-1}\sum_{\ell=1}^n\mathbb{E}\left[\log_\mu(\xi_\ell)\log_\mu(\xi_\ell)^T \right] - vI}\\
&+\lim_{n\to \infty} \norm{ \mathbb{E}\left[\left( \nabla_{\log_\mu(\hat{\mu})} \log_\mu(\xi_\ell) \right)\log_\mu(\xi_\ell)^T+ \log_\mu(\xi_\ell)\left( \nabla_{\log_\mu(\hat{\mu})} \log_\mu(\xi_\ell) \right)^T\right]}\\
&+ \lim_{n\to \infty}\mathbb{E}\left[\mathcal{O}(\dist(\hat{\mu},\mu)^2)\right]\\
\leq& \norm{vI-vI}+ \lim_{n\to \infty} \mathbb{E}\left[\mathcal{O}(\dist(\hat{\mu},\mu))\right]
\end{aligned}.
\]
Here, according to \cite[Proposition 1]{hotz2022central}, it holds that $\dist(\mu,\hat{\mu})$ converges to zero almost surely. Since $M$ is compact, it has finite mass and thus $\dist(\mu,\hat{\mu})$ converges in expectation to 0 by the dominated convergence theorem. Finally, the fact that the empirical covariance has the same rate of convergence as the empirical mean follows from the last inequality. This completes the proof.

\qed

\subsection{Proof of Theorem \ref{thm:average}}
\label{seq:prooflemaverage}
By the very definition, we need to show that
\begin{equation}
\label{eq:arginfintegral2}
\arginf_{q\in \mathbb{S}^2} \int_{\mathbb{S}^2} \dist(q,y)^2 p_{\pr(X)}(y) d\mathbb{S}^2(y) = \pr(\mathbb{E}[X])
\end{equation}
for $X\eqdist N(\mu,\sigma^2 I_3)$. By re-scaling $X$ with a factor of $\frac{1}{\norm{\mu}}$ and by symmetry  a rotation, we may without loss of generality assume $\mu=(-1,0,0)^T$. 

We make a simple first-derivative test to find the minimum in \eqref{eq:arginfintegral2}. Let $(\theta_1,\phi_1)$ be the spherical coordinates that are integrated over the sphere and let $(\theta_2,\phi_2)$ be the coordinates that minimize the integral. Hence we need to show that $(\theta_2,\phi_2) = (\pi,\pi/2)$. Writing the integral in spherical coordinates and using the fact that the derivative with respect to $\phi_2$ is zero, it follows that $(\theta_2,\phi_2)$ satisfies
\[
\frac{\partial}{\partial \phi_2}\int_{\mathbb{S}^2} \dist^2((\theta_1,\phi_1),(\theta_2,\phi_2)) p_{\pr(X)}(\theta_1,\phi_1) \sin(\phi_1) d\theta_1 d\phi_1=0.
\]
Now by Leibniz integral rule the derivative may be moved inside the integral and it holds
\[
\begin{aligned}
\int_0^{2\pi} \int_0^\pi & 2\dist((\theta_1,\phi_1),(\theta_2,\phi_2)) \frac{(\cos(\phi_1)\sin(\phi_2)-\sin(\phi_1)\cos(\phi_2)\cos(\theta_2-\theta_1))}{\sqrt{1-\left(\cos(\phi_1)\cos(\phi_2) + \sin(\phi_1)\sin(\phi_2)\cos(\theta_2-\theta_1)\right)^2}}\\
&p_{\pr(X)}(\theta_1,\phi_1)\sin(\phi_1)  \dd\theta_1 \dd\phi_1=0.
\end{aligned}
\]
By plugging in $(\theta_2,\phi_2)= (\pi,\pi/2)$ into the integral on the left-hand side yields an integral
\[
\begin{aligned}
2\int_0^{2\pi} \int_0^\pi & \acos(\sin(\phi_1)\cos(\pi-\theta_1) ) \frac{\cos(\phi_1)}{\sqrt{1-\sin^2(\phi_1)\cos^2(\pi-\theta_1)}}\\
&p_{\pr(X)}(\theta_1,\phi_1)\sin(\phi_1) \dd\theta_1 \dd\phi_1.
\end{aligned}
\]
This integral equals zero, which can be seen by observing that the integrand is odd along $\phi_1$ around $\frac{\pi}{2}$ since $p_{\pr(X)}$ is symmetric around $\mu$ by construction. 
With similar arguments, we obtain that
\[
\frac{\partial}{\partial \theta_2}\int_{\mathbb{S}^2} \dist^2((\theta_1,\phi_1),(\theta_2,\phi_2)) p_{\pr(X)}(\theta_1,\phi_1) \sin(\phi_1) \dd\phi_1 \dd\theta_1
\]
at $(\theta_2,\phi_2)= (\pi,\pi/2)$ reduces to
\[
2\int_0^{2\pi} \int_0^\pi \acos(\sin(\phi_1)\cos(\pi-\theta_1) ) \frac{\sin^2(\phi_1) \sin(\pi-\theta_1)}{\sqrt{1-\sin^2(\phi_1)\cos^2(\pi-\theta_1)}}p_{\pr(X)}(\theta_1,\phi_1) \dd\phi_1 \dd\theta_1
\]
that equals zero by symmetry around $\theta_1=\pi$. Therefore, one can conclude that $\mu=\pr(\mathbb{E}[X])$ is a local extremum for the integral in Equation \eqref{eq:arginfintegral2}.

Next we shall argue why it is a global minimum. Note that a very similar computation will show that $-\mu$, i.e. $(\theta_2,\phi_2)=(0,\pi/2)$ is another local extremum. In fact, $\mu$ and $-\mu$ are the only local extremum for the integral in Equation \eqref{eq:arginfintegral2} since it is the only two points for which the distance function is rotationally symmetric around the line in $\mathbb{R}^3$ which is spanned by $\mu$. By direct computation
\[
p_{\pr(X)}\left(\pi,\frac{\pi}{2}\right) = \frac{1}{(2\pi)^{3/2}} \exp(-\frac{1}{2\sigma^2}) \left(\frac{1}{\sigma} + \frac{1}{\sigma^2} \frac{\Phi(\frac{1}{\sigma})}{\varphi(\frac{1}{\sigma})}+ \frac{\Phi(\frac{1}{\sigma})}{\varphi(\frac{1}{\sigma})}\right)
\]
and
\[
p_{\pr(X)}\left(0,\frac{\pi}{2}\right) = \frac{1}{(2\pi)^{3/2}} \exp(-\frac{1}{2\sigma^2}) \left(\frac{-1}{\sigma} + \frac{1}{\sigma^2} \frac{\Phi(\frac{-1}{\sigma})}{\varphi(\frac{-1}{\sigma})}+ \frac{\Phi(\frac{-1}{\sigma})}{\varphi(\frac{-1}{\sigma})}\right),
\]
hence $p_{\pr(X)}(\mu)>p_{\pr(X)}(-\mu)$. More generally, a similar computation shows that $p_{\pr(X)}(y)> p_{\pr(X)}(Ry)$, if $\langle y,\mu\rangle >0$ and where $R$ is the reflection mapping over the $y,z$-plane, i.e.
\[
R \begin{pmatrix}
    a\\ b\\ c
\end{pmatrix} = \begin{pmatrix}
    -a \\ b\\ c
\end{pmatrix}.
\]
Tautologically, the distance function increases the farther away one goes, and points farther away will contribute more in the integral. Therefore $(\theta_2,\phi_2)= (\pi,\pi/2)$ yields a global minimum for the integral in Equation \eqref{eq:arginfintegral2}.

\qed

\subsection{Proof of Proposition \ref{prop:bijection}}
\label{sec:bijection}

Without loss of generality, we assume $\mu= (0,0,1)^T$. In this case the functions $A,B,C,\D$ inside Equation \eqref{eq:sphereprobdens} are given by
\[
A= \frac{1}{\sigma^2},
\]
\[
B=\frac{1}{\sigma^2}\cos(\phi),
\]
\[
C= -\frac{1}{2\sigma^2},
\]
and
\[
\D=\frac{1}{\sigma} \cos(\phi).
\]
By \eqref{eq:varianceS2} it holds that $\Cov(\pr(X))$ is isotropic and we can write
\[ 
\begin{aligned}
\frac{\trace(\Cov(\pr(X)))}{2} =& \frac{1}{(2\pi)^{1/2}} \exp(-\frac{1}{2\sigma^2})\\
&\cdot \int_0^\pi  \phi^2\left(\frac{\cos(\phi)}{\sigma} + \left( \frac{\cos^2(\phi)}{\sigma^2} + 1\right)  \frac{\Phi(\frac{\cos(\phi)}{\sigma})}{\varphi(\frac{\cos(\phi)}{\sigma})} \right) \sin(\phi)\dd\phi. 
\end{aligned}
\]
Denote
\[
f(\sigma)= \exp(-\frac{1}{2\sigma^2}) \int_0^\pi  \phi^2\left(\frac{\cos(\phi)}{\sigma} + \left( \frac{\cos^2(\phi)}{\sigma^2} + 1\right)  \frac{\Phi(\frac{\cos(\phi)}{\sigma})}{\varphi(\frac{\cos(\phi)}{\sigma})} \right) \sin(\phi)\dd\phi.
\]
In order to obtain the claim, it suffices to prove that $f(\sigma)$ is strictly increasing from which it follows that $\frac{\trace(\Cov(\pr(X)))}{2}$ is strictly increasing in $\sigma$ as well. For notational simplicity, we set $x=\frac{1}{\sigma}$ and show that $f(x)$ is strictly decreasing in $x$, where now $f(x)$, with a slight abuse of notation, is given by
\begin{equation}
\label{eq:fx}
f(x)=\exp(-\frac{x^2}{2})\int_0^\pi \phi^2\left(x\cos(\phi) + \left( x^2\cos^2(\phi) + 1\right)  \frac{\Phi(x\cos(\phi))}{\varphi(x\cos(\phi))} \right) \sin(\phi)\dd\phi.
\end{equation}
By differentiating in $x$, we get
\[
\begin{aligned}
f'(x)&= \left.\exp(\frac{-x^2}{2})\right[ \\
& -x\int_0^\pi \phi^2\cos(\phi)x\sin(\phi)\dd\phi && =:I_1\\
& +\int_0^\pi \phi^2 \left(-x^3\cos^2(\phi)-x+3x\cos^2(\phi)+x^3\cos^4(\phi)\right) \frac{\Phi(x\cos(\phi))}{\varphi(x\cos(\phi))} \sin(\phi)\dd\phi && =:I_2\\
& +\left. \int_0^\pi \phi^2\left(2\cos(\phi)+x^2\cos^3(\phi)\right)\sin(\phi) \dd\phi\right] &&=:I_3.
\end{aligned}
\]
Immediately, we have that
\[
I_1=x^2\frac{\pi^2}{4}
\]
and
\[
I_3=-\frac{\pi^2}{2}-\frac{5\pi^2}{32}x^2.
\]
In order to show $f'(x)<0$, we need to decipher $I_2$ that is more complicated. The main idea is to show that $f'(x)$ is analytic for all $x\geq 0$ and then show that there is a strictly decreasing analytic function between $f'(x)$ and $0$. We begin with Lemma \ref{lem:Dawson} showing that $\frac{\Phi}{\varphi}$ is a Dawson-like function and therefore analytic. In Lemma \ref{lem:Jterms} the series expansion of $I_2$ is integrated term-wise. The terms of the series expression for $I_2$ is given inductively in Lemma \ref{lem:inductiveform} which are then inserted to give explicit expressions for the terms of $f'(x)\exp(x^2/2)$ in Lemma \ref{lem:constantsoffprim}. Lemma \ref{lem:lowerandincreasing} then gives a lower bound on the odd terms of $f'(x)\exp(x^2/2)$, and Lemma \ref{lem:upperbound} gives an upper bound. By Lemma \ref{lem:decresing} we show that the series expression of $f'(x)\exp(x^2/2)$ is eventually decreasing as a series and in combination with Corollary \ref{cor:alternatingstatement} it is concluded that $f'(x)\exp(x^2/2)$ is entire. The proof is finished by comparing $f'(x)\exp(x^2/2)$ to a linear combination of analytic functions computed in Lemma \ref{lem:simplifiedseriesineq}.

\begin{lemma}
\label{lem:Dawson}
It holds that
\[
\begin{aligned}
\frac{\Phi}{\varphi} (x) = \sum_{k=0}^\infty \frac{1}{(2k+1)!!}x^{2k+1} + \frac{\sqrt{2\pi}}{2}\sum_{k=0}^{\infty} \frac{1}{(2k)!!}x^{2k} =: \sum_{k=0}^\infty d_k x^k
\end{aligned}
\]
where the series converges everywhere.
\end{lemma}

\begin{proof}
Note that $\frac{\Phi}{\varphi} (x)$ is very similar to the Dawson's function (originally studied in \cite{dawson1897numerical}, see also \cite{abramowitz1968handbook} for further details) given by
\[
D_{-}(x)= \exp(x^2)\int_0^x \exp(-t^2)\dd t
\]
and that has series expansion
\[
D_{-}(x) = \sum_{k=0}^\infty \frac{2^k}{(2k+1)!!}x^{2k+1}.
\]
Note that
\[
\begin{aligned}
\int_0^x\exp(-\frac{t^2}{2})\dd t = \sqrt{2}\int_0^{x/\sqrt{2}} \exp(-u^2)\dd u
\end{aligned}
\]
by the variable substitution $u= t/\sqrt{2}$. Hence,
\[
\begin{aligned}
\exp(\frac{x^2}{2}) \int_0^x \exp(-\frac{t^2}{2})\dd t& = \sqrt{2}D_{-}(x/\sqrt{2})\\ 
&= \sqrt{2}\sum_{k=0}^\infty  \frac{2^k}{(2k+1)!!} \frac{x^{2k+1}}{2^k \sqrt{2}}= \sum_{k=0}^\infty \frac{1}{(2k+1)!!} x^{2k+1}.
\end{aligned}
\]
Moreover,
\[
\int_{-\infty}^0 \exp(-\frac{t^2}{2}) \dd t=\frac{\sqrt{2 \pi}}{2}
\]
and hence
\[
\exp(\frac{x^2}{2}) \int_{-\infty}^0 \exp(-\frac{t^2}{2}) \dd t
\]
has series expansion
\[
\frac{\sqrt{2 \pi}}{2} \sum_{k=0}^\infty \frac{1}{k!} \frac{x^{2k}}{2^k}= \frac{\sqrt{2 \pi}}{2} \sum_{k=0}^\infty \frac{1}{(2k)!!} x^{2k}.
\]
It follows that
\[
\begin{aligned}
\frac{\Phi}{\varphi} (x)&= \exp(\frac{x^2}{2}) \int_{-\infty}^0 \exp(-\frac{t^2}{2}) \dd t + \exp(\frac{x^2}{2}) \int_{0}^x \exp(-\frac{t^2}{2}) \dd t\\
&=\sum_{k=0}^\infty \frac{1}{(2k+1)!!}x^{2k+1} + \frac{\sqrt{2\pi}}{2}\sum_{k=0}^{\infty} \frac{1}{(2k)!!}x^{2k}
\end{aligned}
\]
proving the claimed series representation. Finally, the convergence everywhere follows from the fact that Dawson's function converges everywhere.
\end{proof}
Using Lemma \ref{lem:Dawson} it follows that $I_2$ can be rewritten as
\begin{equation}
\label{eq:seriesofI2}
I_2 = \sum_{k=0}^\infty d_k \int_0^\pi  \phi^2 \left(-x^3\cos^2(\phi)-x+3x\cos^2(\phi)+x^3\cos^4(\phi)\right)x^k\cos^k(\phi) \sin(\phi) \dd \phi.
\end{equation}
Each term in Equation \eqref{eq:seriesofI2} involves
\[
J_m=\int_0^\pi \phi^2 \cos^m(\phi) \sin(\phi)\dd\phi
\]
that we study next.
\begin{lemma}
\label{lem:Jterms}
The $J_m$ sequence satisfies 
$J_0=\pi^2-4$ and for even $m\neq 0$
\[
\frac{1}{m+1}\left(\pi^2-4\frac{m!!}{(m+1)!!}\left( 1+ \sum_{j=0}^{\frac{m}{2}-1} \frac{1}{2^{2j+1}(2j+3)}{2j+1 \choose j}\right)\right),
\]
while for $m\in \mathbb{N}$ odd we have 
\[
J_m = \frac{\pi^2}{m+1}\left(\frac{m!!}{(m+1)!!} -1\right).
\]
\end{lemma}

\begin{proof}
For $J_0$ we observe immediately
\[
J_0= \int_0^\pi \phi^2 \sin(\phi) \dd\phi= \pi^2-4.
\]
Let next $m$ be odd. Then integration by parts gives
\[
\begin{aligned}
J_m=& \left[ -\phi^2 \frac{\cos^{m+1}(\phi)}{m+1} \right]_0^\pi + \frac{2}{m+1} \int_0^\pi\phi \cos^{m+1}(\phi)\dd\phi\\
=& -\frac{\pi^2}{m+1} + \frac{2}{m+1}\left[\frac{m!!}{(m+1)!!}\phi^2+\phi\sum_{j=0}^{\frac{m-1}{2}} \cos^{m-2j}(\phi) \sin(\phi)\frac{m!!}{(m-2j)!!} \right.\\
& \cdot\left. \frac{(m-2j-1)!!}{(m+1)!!} \right]_0^\pi - \frac{2}{m+1} \sum_{j=0}^{\frac{m-1}{2}}\frac{m!!}{(m-2j)!!}\frac{(m-2j-1)!!}{(m+1)!!}\\
&\cdot \int_0^\pi \cos^{m-2j}(\phi) \sin(\phi)d\phi -\frac{2}{m+1}\frac{m!!}{(m+1)!!} \int_0^\pi \phi \dd\phi\\
&= -\frac{\pi^2}{m+1}+\frac{2}{m+1}\frac{m!!}{(m+1)!!}\pi^2- \frac{1}{m+1} \frac{m!!}{(m+1)!!} \pi^2\\
&= \frac{\pi^2}{m+1}\left(\frac{m!!}{(m+1)!!} -1\right).
\end{aligned}
\]
Similarly, when $m>0$ is even, integration by parts gives
\[
\begin{aligned}
J_m =& \left[ -\phi^2 \frac{\cos^{m+1}(\phi)}{m+1} \right]_0^\pi + \frac{2}{m+1} \int_0^\pi\phi \cos^{m+1}(\phi)\dd\phi\\
=& \frac{\pi^2}{m+1} + \frac{2}{m+1} \left[\phi \sin(\phi)\frac{m!!}{(m+1)!!}+ \phi\sum_{j=0}^{\frac{m}{2}-1}\cos^{m-2j}\sin(\phi) \frac{m!!}{(m-2j)!!}\right.\\
&\cdot\left. \frac{(m-1-2j)!!}{(m+1)!!}\right]_0^\pi- \frac{2}{m+1} \frac{m!!}{(m+1)!!} \int_0^\pi \sin(\phi)d\phi\\
& - \frac{2}{m+1} \sum_{j=0}^{\frac{m}{2}-1} \frac{m!!}{(m-2j)!!} \frac{(m-1-2j)!!}{(m+1)!!} \int_0^\pi \cos^{m-2j}(\phi)\sin(\phi)\dd\phi\\
=& \frac{\pi^2}{m+1} - \frac{4}{m+1}\frac{m!!}{(m+1)!!}-\frac{4}{m+1}\frac{m!!}{(m+1)!!} \sum_{j=0}^{\frac{m}{2}-1}\frac{(m-1-2j)!!}{(m-2j)!! (m+1-2j)}\\
=&\frac{1}{m+1}\left(\pi^2-4\frac{m!!}{(m+1)!!}\left( 1+ \sum_{j=0}^{\frac{m}{2}-1} \frac{1}{2^{2j+1}(2j+3)}{2j+1 \choose j}\right)\right).
\end{aligned} 
\]
This completes the proof.
\end{proof}
\begin{lemma}
\label{lem:inductiveform}
Denote
\[
I_2= \sum_{n=1}^\infty a_n x^n.
\]
Then the coefficients $a_k$ satisfy
\[
a_1 = \frac{4\sqrt{2\pi}}{9}, \quad a_2 =-\frac{7\pi^2}{32},
\]
and
\[
a_k= d_{k-1}(3J_{k+1}-J_{k-1})+ d_{k-3}(J_{k+1}-J_{k-1}),\quad k\geq 3.
\]

\end{lemma}

\begin{proof}
Immediately $I_2(0)=0$ so the constant term is zero. Next, we get expressions for $a_n$ in terms of $d_n$:s and $J_n$:s from Equation \eqref{eq:seriesofI2}. Then straightforward calculations give
\[
a_1 =d_0(-J_0+3J_2)=\frac{\sqrt{2\pi}}{2}\left( 4-\pi^2 + \frac{3}{3}\left(\pi^2 -4\frac{2}{3}\left(1+\frac{1}{6}\right)\right)\right)= \frac{4\sqrt{2\pi}}{9},
\]
\[
a_2= d_1(-J_1+3J_3)= \frac{\pi^2}{4} + \frac{3\pi^2}{4}\left(\frac{3}{8}-1\right)=-\frac{7\pi^2}{32},
\]
and, for $k\geq 3$,
\[
a_k= d_{k-1}(3J_{k+1}-J_{k-1})+ d_{k-3}(J_{k+1}-J_{k-1}).
\]

\end{proof}
\begin{lemma}
\label{lem:constantsoffprim}
Denote
\begin{equation}
\label{eq:c_k}
f'(x)\exp(\frac{x^2}{2})= \sum_{k=0} c_k x^k.
\end{equation}
Then
\[
c_0=-\frac{\pi^2}{2},\quad c_1 =\frac{4\sqrt{2\pi}}{9}, \quad c_2=- \frac{\pi^2}{8},
\]
for $k\geq 3$ odd
\[
c_k=\frac{2\sqrt{2\pi}}{(k+2)!!}\left(1 + \sum_{j=0}^{\frac{k-1}{2}-1} \frac{1}{2^{2j+1}(2j+3)}{2j+1 \choose j}- \frac{k+1}{2^{k}(k+2)} {k \choose \frac{k-1}{2}} \right),
\]
and for $k\geq 4$ even
\[
c_k= -\frac{\pi^2}{(k+2)!!}.
\]
\end{lemma}

\begin{proof}
By considering expressions for $I_1$, $I_2$, and $I_3$ gives
\[
c_0=-\frac{\pi^2}{2},
\]
\[
c_1=a_1 =\frac{4\sqrt{2\pi}}{9},
\]
\[
c_2=a_2+\frac{3 \pi^2}{32} =- \frac{\pi^2}{8},
\]
and 
\[
c_k=a_k \qquad \text{ whenever } k\geq 3,
\]
where $a_k$ is given in Lemma \ref{lem:inductiveform}. It follows that for $k\geq 3$ odd we have
\[
\begin{aligned}
c_k=&d_{k-1}(3J_{k+1}-J_{k-1})+ d_{k-3}(J_{k+1}-J_{k-1})\\
=& \frac{\sqrt{2\pi}}{2} \left(\frac{1}{(k-1)!!}\left[\frac{3}{k+2}\left(\pi^2-4 \frac{(k+1)!!}{(k+2)!!}\left( 1+\sum_{j=0}^{\frac{k+1}{2}-1}\frac{1}{2^{2j+1}(2j+3)}{2j+1 \choose j} \right)\right.\right.\right.\\
&\left.-\frac{1}{k} \left( \pi^2-4\frac{(k-1)!!}{k!!}\left( 1+\sum_{j=0}^{\frac{k-1}{2}-1}\frac{1}{2^{2j+1}(2j+3)}{2j+1 \choose j} \right) \right)\right]\\
&+\frac{1}{(k-3)!!}\left[\frac{1}{k+2}\left(\pi^2-4 \frac{(k+1)!!}{(k+2)!!}\left( 1+\sum_{j=0}^{\frac{k+1}{2}-1}\frac{1}{2^{2j+1}(2j+3)}{2j+1 \choose j} \right).\right.\right.\\
&\left.\left.-\frac{1}{k} \left( \pi^2-4\frac{(k-1)!!}{k!!}\left( 1+\sum_{j=0}^{\frac{k-1}{2}-1}\frac{1}{2^{2j+1}(2j+3)}{2j+1 \choose j} \right) \right)\right]\right)\\
=&  \frac{2\sqrt{2\pi}}{(k+2)!!}\left(1 + \sum_{j=0}^{\frac{k-1}{2}-1} \frac{1}{2^{2j+1}(2j+3)}{2j+1 \choose j}- \frac{k+1}{2^{k}(k+2)} {k \choose \frac{k-1}{2}} \right).
\end{aligned}
\]
Similarly for $k\geq 4$ even we get
\[
\begin{aligned}
c_k= & d_{k-1}(3J_{k+1}-J_{k-1})+ d_{k-3}(J_{k+1}-J_{k-1})\\
=& \frac{\pi^2}{(k-1)!!}\left[\frac{3}{k+2}\left( \frac{(k+1)!!}{(k+2)!!}-1\right) - \frac{1}{k}\left(\frac{(k-1)!!}{k!!}-1 \right)\right]\\
&+ \frac{\pi^2}{(k-3)!!}\left[ \frac{1}{k+2}\left( \frac{(k+1)!!}{(k+2)!!}-1\right) - \frac{1}{k}\left(\frac{(k-1)!!}{k!!}-1 \right)\right]\\
=& -\frac{\pi^2}{(k+2)!!}.
\end{aligned}
\]
This completes the proof.
\end{proof}
\begin{lemma}
\label{lem:lowerandincreasing}
For all $k\geq 3$ odd, set
\begin{equation}
\label{eq:Sk}
S(k)=  1 + \sum_{j=0}^{\frac{k-1}{2}-1} \frac{1}{2^{2j+1}(2j+3)}{2j+1 \choose j}- \frac{k+1}{2^{k}(k+2)} {k \choose \frac{k-1}{2}}.
\end{equation}
Then $S(k)$ is increasing and satisfies $2\sqrt{2\pi}S(k) \geq \pi$.
\end{lemma}
\begin{proof}
First note that 
\[
2\sqrt{2\pi}S(3) = 2\sqrt{2\pi}\left( 1 + \frac{1}{6} {1 \choose 0} - \frac{4}{5}\frac{1}{2^3} {3 \choose 1} \right)= 2\sqrt{2\pi}\left( 1+ \frac{1}{6} - \frac{3}{10} \right)= \frac{26\sqrt{2\pi}}{15}\geq \pi.
\]
It remains to show that $S(k)$ is increasing. We have
\[
\begin{aligned}
S(k+2) &= 1 + \sum_{j=0}^{\frac{k-1}{2}} \frac{1}{2^{2j+1}(2j+3)}{2j+1 \choose j}- \frac{k+3}{2^{k+2}(k+4)} {k+2 \choose \frac{k+1}{2}} \\
&= S(k) + \frac{k+1}{2^k(k+2)} {k \choose \frac{k-1}{2}}+ \frac{1}{2^k(k+2)} {k \choose \frac{k-1}{2}} - \frac{k+3}{2^{k+2}(k+4)}{k+2 \choose \frac{k+1}{2}}\\
&= S(k) + \frac{1}{2^k} {k \choose \frac{k-1}{2}} - \frac{k+3}{2^{k+2}(k+4)} {k+2 \choose \frac{k+1}{2}}\\
&\geq  S(k) + \frac{1}{2^k} {k \choose \frac{k-1}{2}} - \frac{1}{2^{k+2}} {k+2 \choose \frac{k+1}{2}}\\
&= S(k) + \frac{(k+1)!}{2^{k+2} \left(\frac{k+1}{2}\right)! \left(\frac{k+3}{2}\right)!} \left(k+3-(k+2) \right) \\
&= S(k) + \frac{(k+1)!}{2^{k+2} \left(\frac{k+1}{2}\right)! \left(\frac{k+3}{2}\right)!}
\end{aligned}
\]
and hence $S(k+2)\geq S(k)$. This completes the proof.
\end{proof}

\begin{cor}
\label{cor:alternatingstatement}
For $x> 0$ the series
\[
\sum_{k=0}^\infty c_k x^k
\]
is alternating, where the coefficients $c_k$ are determined by \eqref{eq:c_k}.
\end{cor}
\begin{proof}
For $k$ even we clearly have $c_k<0$. For $k$ odd on the other hand, $c_k$ is positive since
\[
c_k\geq \frac{\pi}{(k+2)!!}
\]
by Lemma \ref{lem:lowerandincreasing}.

\end{proof}
The following result shows that the sum in Lemma \ref{lem:lowerandincreasing} can be bounded from above as well.
\begin{lemma}
\label{lem:upperbound}
For all $k\geq 3$ odd, $S(k)$ defined by \eqref{eq:Sk} satisfies $2\sqrt{2\pi}S(k)\leq \pi^2$.
\end{lemma}
\begin{proof}
By elementary manipulations as in the proof of Lemma \ref{lem:lowerandincreasing} we obtain
\[
S(k+2)=S(k)+ \frac{2\sqrt{2\pi}(k+3)}{2^{k+2}(k+2)(k+4)} {k+2 \choose \frac{k+1}{2}}.
\]
From Stirling's approximation we infer
\[
{k+2 \choose \frac{k+1}{2}} \leq \sqrt{\frac{2}{\pi}}\frac{ 2^{k+2}}{\sqrt{(k+2)}},
\]
and hence
\[
S(k+2)\leq S(k) + \frac{\sqrt{2}}{\sqrt{\pi}(k+2)^{3/2}}.
\]
Moreover, 
\[
S(3)= 1+\frac{1}{6} - \frac{3}{10} \leq 1 + \frac{1}{3^{3/2}}
\]
allows us to estimate
\[
\lim_{k\to \infty} 2\sqrt{2\pi} S(k) \leq 4 \sum_{k=0}^{\infty} \frac{1}{(2k+1)^{3/2}} \leq \pi^2.
\]
Since $S$ is increasing by Lemma \ref{lem:lowerandincreasing}, it follows that $S(k)\leq \pi^2$ for all $k\geq 3$ odd. This completes the proof.
\end{proof}
\begin{lemma}
\label{lem:decresing}
Let the coefficients $c_k,k=1,2,\ldots$ be given by \eqref{eq:c_k} and let $x>0$ be fixed. Then the terms in the series
\[
\sum_{k=M}^\infty c_k x^k
\]
decreases monotonically for large enough $M$.
\end{lemma}

\begin{proof}
Consider first the case when $k$ is odd. Then using the upper bound of Lemma \ref{lem:upperbound} yields
\[
\abs{\frac{c_{k+1}x^{k+1}}{c_k x^k}} \leq \frac{\frac{\pi^2}{(k+3)!!} }{\frac{\pi^2}{(k+2)!!}}x = \frac{(k+2)!!}{(k+3)!!}x
\]
that is less than one for sufficiently large $k$ (depending on $x$). Similarly for $k$ even we can use the lower bound from Lemma \ref{lem:lowerandincreasing} to obtain
\[
\abs{\frac{c_{k+1}x^{k+1}}{c_k x^k}} \leq \frac{\frac{\pi^2}{(k+3)!!} }{\frac{\pi}{(k+2)!!}}x = \frac{(k+2)!!}{(k+3)!!} \pi x
\]
that again is less than one for large enough $k$. This yields the claim.
\end{proof}

\begin{lemma}
\label{lem:simplifiedseriesineq}
Denote
\[
M(x)= \sum_{k=0}^\infty \frac{x^{2k}}{(2k+2)!!}
\]
and
\[
N(x) = \sum_{k=0}^\infty \frac{x^{2k+1}}{(2k+3)!!}.
\]
Then
\[
M(x) > N(x)
\]
for all $x\geq 0$.
\end{lemma}
\begin{proof}
By differentiating we get
\[
\begin{aligned}
M'(x)&= \sum_{k=0}^{\infty} \frac{(2k+2)x^{2k+1}}{(2k+4)!!} = \sum_{k=0}^\infty \frac{x^{2k+1}}{(2k+3)!!}  \frac{(k+1)}{2^{2k+2}} {2k+3 \choose k+1}.
\end{aligned}
\]
Now using
\[
{2k+3 \choose k+1} \geq \frac{2^{2k+2}}{\sqrt{k+1}}
\]
gives
\[
M'(x) \geq \sum_{k=0}^\infty \frac{\sqrt{k+1}}{(2k+3)!!} x^{2k+1} \geq N(x).
\]
Similarly it holds that
\[
N'(x) =\sum_{k=0}^\infty \frac{(2k+1)x^{2k}}{(2k+3)!!} = \sum_{k=0}^\infty \frac{x^{2k}}{(2k+2)!!} \frac{(2k+1)2^{2k+2}}{(k+2){ 2k+3 \choose k+1}}.
\]
In this case we can use 
\[
{2k+3 \choose k+1} \geq \frac{2^{2k+3}}{2k+3}
\]
leading to
\[
N'(x) \leq \sum_{k=0}^\infty \frac{x^{2k}}{(2k+2)!!} \frac{k+\frac{1}{2}}{(k+2)(2k+3)} \leq M(x).
\]
Combining the two bounds above gives us 
\[
\ddf{ }{ }{x} \left(M^2(x) - N^2(x) \right) = 2M(x) M'(x)-2 N(x) N'(x) \geq M(x)N(x) - M(x)N(x) =0.
\]
Consequently, $M^2-N^2$ is an increasing function for $x\geq 0$ leading to
\[
M^2(x)-N^2(x)\geq M^2(0)-N^2(0) =1.
\]
It follows that
\[
M(x) \geq \sqrt{1+N^2(x)}> N(x)
\]
and the proof is complete.
\end{proof}
We are finally in the position to prove Proposition \ref{prop:bijection}.
\begin{proof}[Proof of Proposition \ref{prop:bijection}.]
By Corollary \ref{cor:alternatingstatement} and Lemma \ref{lem:decresing} the series expansion \eqref{eq:c_k} for $f'(x)\exp(x^2/2)$ is convergent by Leibniz alternating series test. Since the radius of convergence is unbounded, it is analytic and thus we may split the series into its positive part and negative part given by
\[
\sum_{k=0}^\infty  c_{2k} x^{2k} =: Q(x), \qquad \sum_{k=0}^\infty c_{2k+1} x^{2k+1} =: P(x).
\] 
By Lemma \ref{lem:constantsoffprim} it holds that $Q(x)=-\pi^2 M(x)$ where $M(x)$ is as in Lemma \ref{lem:simplifiedseriesineq}. Now Lemma \ref{lem:upperbound} gives
\[
c_{2k} \leq \frac{\pi^2}{(k+2)!!},
\]
and hence
\[
Q(x) \leq \pi^2 N(x),
\]
where $N(x)$ is as in Lemma \ref{lem:simplifiedseriesineq}.
Therefore
\[
f'(x)\exp(x^2/2) = Q(x) + P(X) \leq - \pi^2M(x) + \pi^2N(x).
\]
Applying Lemma \ref{lem:simplifiedseriesineq} to the above inequality now gives
\[
f'(x)\exp(\frac{x^2}{2}) < 0,
\]
and thus $f'(x)<0$ for all $x\geq 0$, where $f$ is given by \eqref{eq:fx}. The claim follows from this. 
\end{proof}

\bibliographystyle{plain}
\bibliography{refs.bib}

\begin{thebibliography}{10}

\bibitem{abramowitz1968handbook}
M.~Abramowitz and I.~A. Stegun.
\newblock {\em Handbook of mathematical functions with formulas, graphs, and
  mathematical tables}, volume~55.
\newblock US Government printing office, 1968.

\bibitem{ahidar2020convergence}
A.~Ahidar-Coutrix, T.~Le~Gouic, and Q.~Paris.
\newblock Convergence rates for empirical barycenters in metric spaces:
  curvature, convexity and extendable geodesics.
\newblock {\em Probability theory and related fields}, 177(1):323--368, 2020.

\bibitem{chang1993spherical}
Ted Chang.
\newblock Spherical regression and the statistics of tectonic plate
  reconstructions.
\newblock {\em International Statistical Review/Revue Internationale de
  Statistique}, pages 299--316, 1993.

\bibitem{dawson1897numerical}
H.~G. Dawson.
\newblock On the numerical value of $\int_0^h e^{x^2} dx$.
\newblock {\em Proceedings of the London Mathematical Society}, 1(1):519--522,
  1897.

\bibitem{Fisher-Lewis-1993}
N.~I. Fisher, T.~Lewis, and B.~J.~J. Embleton.
\newblock {\em Statistical analysis of spherical data}.
\newblock Cambridge University Press, 1993.

\bibitem{gao2007evidence}
F.~Gao, K.~Chia, and D.~Machin.
\newblock On the evidence for seasonal variation in the onset of acute
  lymphoblastic leukemia (all).
\newblock {\em Leukemia research}, 31(10):1327--1338, 2007.

\bibitem{hernandez2017general}
D.~Hernandez-Stumpfhauser, F.~J. Breidt, and M.~J. {van der Woerd}.
\newblock The general projected normal distribution of arbitrary dimension:
  Modeling and bayesian inference.
\newblock {\em Bayesian Analysis}, 12(1):113--133, 2017.

\bibitem{hotz2022central}
T.~Hotz, H.~Le, and A.~T.~A. Wood.
\newblock Central limit theorem for intrinsic frechet means in smooth compact
  riemannian manifolds.
\newblock {\em Probab. Theory Relat. Fields}, 2024.

\bibitem{hussien2014multi}
M.~T. Hussien, K.~G. Seddik, R.~H. Gohary, M.~Shaqfeh, H.~Alnuweiri, and
  H.~Yanikomeroglu.
\newblock Multi-resolution broadcasting over the grassmann and stiefel
  manifolds.
\newblock In {\em 2014 IEEE International Symposium on Information Theory},
  pages 1907--1911. IEEE, 2014.

\bibitem{kells1940plane}
L.~M. Kells, W.~F. Kern, and J.~R. Bland.
\newblock {\em Plane and Spherical Trigonometry}.
\newblock US Armed Forces Institute, 1940.

\bibitem{lang2012mpg}
M.~Lang and W.~Feiten.
\newblock Mpg-fast forward reasoning on 6 dof pose uncertainty.
\newblock In {\em ROBOTIK 2012; 7th German Conference on Robotics}, pages 1--6.
  VDE, 2012.

\bibitem{levitt1976simplified}
Michael Levitt.
\newblock A simplified representation of protein conformations for rapid
  simulation of protein folding.
\newblock {\em Journal of molecular biology}, 104(1):59--107, 1976.

\bibitem{li2018wavefront}
K.~Li, D.~Frisch, S.~Radtke, B.~Noack, and U.~D. Hanebeck.
\newblock Wavefront orientation estimation based on progressive bingham
  filtering.
\newblock In {\em 2018 Sensor Data Fusion: Trends, Solutions, Applications
  (SDF)}, pages 1--6. 2018.

\bibitem{nunez2015bayesian}
G.~Nu{\~n}ez-Antonio, M.~C. Aus{\'\i}n, and M.~P. Wiper.
\newblock Bayesian nonparametric models of circular variables based on
  dirichlet process mixtures of normal distributions.
\newblock {\em Journal of Agricultural, Biological, and Environmental
  Statistics}, 20:47--64, 2015.

\bibitem{Nunez-Gutierrez-Escarela-2011}
G.~Nu{\~n}ez-Antonio, E.~Guti{\'e}rrez-Pe{\~n}a, and G.~Escarela.
\newblock A {B}ayesian regression model for circular data based on the
  projected normal distribution.
\newblock {\em Statistical Modeling}, 11:185--201, 2011.

\bibitem{Nunez-Gutierrez-2005}
G.~Nuñez-Antonio and E.~Gutiérrez-Peña.
\newblock A {B}ayesian analysis of directional data using the projected normal
  distribution.
\newblock {\em Journal of Applied Statistics}, 32:995--1001, 2005.

\bibitem{pennec2012exponential}
X.~Pennec and V.~Arsigny.
\newblock Exponential barycenters of the canonical cartan connection and
  invariant means on lie groups.
\newblock In {\em Matrix information geometry}, pages 123--166. Springer, 2012.

\bibitem{pennec2006intrinsic}
Xavier Pennec.
\newblock Intrinsic statistics on riemannian manifolds: Basic tools for
  geometric measurements.
\newblock {\em Journal of Mathematical Imaging and Vision}, 25:127--154, 2006.

\bibitem{pitaval2013joint}
R.~Pitaval and O.~Tirkkonen.
\newblock Joint grassmann-stiefel quantization for mimo product codebooks.
\newblock {\em IEEE transactions on wireless communications}, 13(1):210--222,
  2013.

\bibitem{Presnell-etal-1998}
B.~Presnell, S.~P. Morrison, and R.~C. Littell.
\newblock Projected multivariate linear models for directional data.
\newblock {\em Journal of the American Statistical Association}, 93:1068--1077,
  1998.

\bibitem{seddik2017multi}
K.~G. Seddik, R.~H. Gohary, M.~T. Hussien, M.~Shaqfeh, H.~Alnuweiri, and
  H.~Yanikomeroglu.
\newblock Multi-resolution multicasting over the grassmann and stiefel
  manifolds.
\newblock {\em IEEE Transactions on Wireless Communications}, 16(8):5296--5310,
  2017.

\bibitem{sveier2018pose}
A.~Sveier and O.~Egeland.
\newblock Pose estimation using dual quaternions and moving horizon estimation.
\newblock {\em IFAC-PapersOnLine}, 51(13):186--191, 2018.

\bibitem{Tanabe-etal-2007}
A.~Tanabe, K.~Fukumizu, S.~Oba, T.~Takenouchi, and S.~Ishii.
\newblock Parameter estimation for von mises–fisher distributions.
\newblock {\em Computational Statistics}, 22:145--157, 2007.

\bibitem{Wang-Gelfand-2013}
F.~Wang and A.~E. Gelfand.
\newblock Directional data analysis under the general projected normal
  distribution.
\newblock {\em Statistical Methodology}, 10(1):113--127, 2013.

\end{thebibliography}
\end{document}